\documentclass[10pt,draft]{article}
\usepackage{amsmath,amscd}
\usepackage{amssymb,latexsym,amsthm}
\usepackage{color}
\usepackage[spanish,english]{babel}
\usepackage{amssymb}
\usepackage{color}
\usepackage{amsmath,amsthm,amscd}
\usepackage[latin1]{inputenc}
\usepackage{lscape}
\usepackage{fancyhdr}
\usepackage{amsfonts}
\usepackage{pb-diagram}
\usepackage{hyperref}
\usepackage{float}
\numberwithin{equation}{section}
\newtheorem{theorem}{Theorem}[section]

\newtheorem{definition}[theorem]{Definition}
\newtheorem{proposition}[theorem]{Proposition}

\theoremstyle{definition}

\newtheorem{example}[theorem]{Example}

\newtheorem{remark}[theorem]{Remark}

\title{\textbf{The center of the total ring of fractions}}
\author{Oswaldo Lezama\\
\texttt{jolezamas@unal.edu.co}
\\Helbert Venegas\\
\texttt{hjvenegasr@unal.edu.co}
\\ Seminario de Álgebra Constructiva - SAC$^2$\\ Departamento de Matemáticas\\ Universidad Nacional de
Colombia, Sede Bogotá}
\date{}
\begin{document}
\maketitle
\begin{abstract}
\noindent Let $A$ be a right Ore domain, $Z(A)$ be the center of
$A$ and $Q_r(A)$ be the right total ring of fractions of $A$. If
$K$ is a field and $A$ is a $K$-algebra, in this short paper we
prove that if $A$ is finitely generated and ${\rm GKdim}(A)<{\rm
GKdim}(Z(A))+1$, then $Z(Q_r(A))\cong Q(Z(A))$. Many examples that
illustrate the theorem are included, most of them within the skew
$PBW$ extensions.

\bigskip


\bigskip

\end{abstract}


\section{Introduction}

Given an Ore domain $A$, it is interesting to know when
$Z(Q(A))\cong Q(Z(A))$, where $Z(A)$ is the center of $A$ and
$Q(A)$ is the total ring of fractions of $A$. This question became
important after the formulation of the Gelfand-Kirillov conjecture
in \cite{GK}: \textit{Let $\mathcal{G}$ be an algebraic Lie
algebra of finite dimension over a field $K$, with $char(K)=0$.
Then, there exist integers $n,k\geq 1$ such that
\begin{equation}\label{GKconjecture}
Q(\mathcal{U}(\mathcal{G}))\cong Q(A_n(K[s_1,\dots,s_k])),
\end{equation}
where $\mathcal{U}(\mathcal{G})$ is the enveloping algebra of
$\mathcal{G}$ and $A_n(K[s_1,\dots,s_k])$ is the general Weyl
algebra over $K$}. In the investigation of this famous conjecture
the isomorphism between the center of the total ring of fractions
and the total ring of fractions of the center occupies a special
key role. There are remarkable examples of algebras for which the
conjecture holds and they satisfy the isomorphism. For example, if
$\mathcal{G}$ is a finite dimensional nilpotent Lie algebra over a
field $K$, with ${\rm char}(K)=0$, then the conjecture holds and
$Z(Q(\mathcal{U}(\mathcal{G})))\cong
Q(Z(\mathcal{U}(\mathcal{G})))$ (\cite{GK}, Lemma 8). More
recently, the quantum version of the Gelfand-Kirillov conjecture
has occupied the attention of many researchers. One example of
this is the following (see \cite{Alev2}, Theorem 2.15): Let
$U_q^{+}(sl_m)$ be the quantum enveloping algebra of the Lie
algebra of strictly superior triangular matrices of size $m\times
m$, $m\geq 3$, over a field $K$. If $m=2n+1$, then
$Q(U_q^{+}(sl_m))\cong Q({\rm K}_{{\rm q}}[x_1,\dots, x_{2n^2}])$,
where ${\rm K}_{{\rm q}}[x_1,\dots, x_{2n^2}]$ is the quantum ring
of polynomials, ${\rm K}:=Q(Z(U_q^{+}(sl_m)))$ and ${\rm
q}:=[q_{ij}]\in M_{2n^2}(K)$, with $q_{ii}=1=q_{ij}q_{ji}$, and
$q_{ij}$ is a power of $q$ for every $1\leq i,j\leq 2n^2$. If
$m=2n$, then $Q(U_q^{+}(sl_m))\cong Q({\rm K}_{{\rm q}}[x_1,\dots,
x_{2n(n-1)}])$, where ${\rm q}:=[q_{ij}]\in M_{2n(n-1)}(K)$, with
$q_{ii}=1=q_{ij}q_{ji}$, and $q_{ij}$ is a power of $q$ for every
$1\leq i,j\leq 2n(n-1)$. Moreover, in both cases
$Z(Q(U_q^{+}(sl_m)))\cong Q(Z(U_q^{+}(sl_m)))$.

Let $A$ be a right Ore domain. In this paper we study the
isomorphism $Z(Q_r(A))\cong Q(Z(A))$, where $Z(A)$ is the center
of $A$ and $Q_r(A)$ is the right total ring of fractions of $A$.
The main tool that we will use is the Gelfand-Kirillov dimension,
so we will assume that $A$ is a $K$-algebra, where $K$ is an
arbitrary field. The principal result is Theorem
\ref{theorem6.4.2} proved in Section 2. The result can also be
interpreted as a way of computing the center of $Q_r(A)$. In
Section 3 we include many examples that illustrate the theorem,
most of them within the skew $PBW$ extensions.

We start with the following known facts about skew $PBW$
extensions that will be used in the examples.

\begin{definition}[\cite{LezamaGallego}]\label{gpbwextension}
Let $R$ and $A$ be rings. We say that $A$ is a \textit{skew $PBW$
extension of $R$} $($also called a $\sigma-PBW$ extension of
$R$$)$ if the following conditions hold:
\begin{enumerate}
\item[\rm (i)]$R\subseteq A$.
\item[\rm (ii)]There exist finitely many elements $x_1,\dots ,x_n\in A$ such $A$ is a left $R$-free module with basis
\begin{center}
${\rm Mon}(A):= \{x^{\alpha}=x_1^{\alpha_1}\cdots
x_n^{\alpha_n}\mid \alpha=(\alpha_1,\dots ,\alpha_n)\in
\mathbb{N}^n\}$, with $\mathbb{N}:=\{0,1,2,\dots\}$.
\end{center}
The set ${\rm Mon}(A)$ is called the set of standard monomials of
$A$.
\item[\rm (iii)]For every $1\leq i\leq n$ and $r\in R-\{0\}$ there exists $c_{i,r}\in R-\{0\}$ such that
\begin{equation}\label{sigmadefinicion1}
x_ir-c_{i,r}x_i\in R.
\end{equation}
\item[\rm (iv)]For every $1\leq i,j\leq n$ there exists $c_{i,j}\in R-\{0\}$ such that
\begin{equation}\label{sigmadefinicion2}
x_jx_i-c_{i,j}x_ix_j\in R+Rx_1+\cdots +Rx_n.
\end{equation}
Under these conditions we will write $A:=\sigma(R)\langle
x_1,\dots ,x_n\rangle$.
\end{enumerate}
\end{definition}
Associated to a skew $PBW$ extension $A=\sigma(R)\langle x_1,\dots
,x_n\rangle$ there are $n$ injective endomorphisms
$\sigma_1,\dots,\sigma_n$ of $R$ and $\sigma_i$-derivations, as
the following proposition shows.

\begin{proposition}[\cite{LezamaGallego}, Proposition 3]\label{sigmadefinition}
Let $A$ be a skew $PBW$ extension of $R$. Then, for every $1\leq
i\leq n$, there exist an injective ring endomorphism
$\sigma_i:R\rightarrow R$ and a $\sigma_i$-derivation
$\delta_i:R\rightarrow R$ such that
\begin{center}
$x_ir=\sigma_i(r)x_i+\delta_i(r)$,
\end{center}
for each $r\in R$.
\end{proposition}
A particular case of skew $PBW$ extension is when all $\sigma_i$
are bijective and the constants $c_{ij}$ are invertible.

\bigskip

\begin{definition}[\cite{LezamaGallego}]\label{sigmapbwderivationtype}
Let $A$ be a skew $PBW$ extension. $A$ is bijective if $\sigma_i$
is bijective for every $1\leq i\leq n$ and $c_{i,j}$ is invertible
for any $1\leq i<j\leq n$.
\end{definition}

\begin{proposition}[\cite{lezamaore}, Theorem 3.6]\label{Orespbw}
Let $A=\sigma(R)\langle x_1,\dots, x_n\rangle$ be a bijective skew
$PBW$ extension of a right Ore domain $R$. Then $A$ is also a
right Ore domain.
\end{proposition}

\begin{proposition}[\cite{Reyes5}, Theorem 14]\label{lezamareyes5.1}
Let $R$ be a $K$-algebra with a finite dimensional generating
subspace $V$ and let $A=\sigma(R)\langle x_1,\dotsc,x_n\rangle$ be
a bijective skew $PBW$ extension of $R$. If $\sigma_i,\delta_i$
are $K$-linear and $\sigma_i(V)\subseteq V$, for $1\leq i\leq n$,
then
\[
{\rm GKdim}(A)={\rm GKdim}(R)+n.
\]
\end{proposition}

\begin{proposition}\label{proposition1.2.8}
Let $R$ be a commutative domain, $\sigma$ an automorphism of $R$
and
\begin{center}
$R^\sigma:=\{r\in R|\sigma(r)=r\}$.
\end{center}
If $\sigma$ has infinite order,
then $Z(R[x;\sigma])=R^\sigma$. If $\sigma$ has finite order $v$,
then $Z(R[x;\sigma])=R^\sigma[x^v]$.
\end{proposition}
\begin{proof}
The proof when $R$ is a field can be found in \cite{Rowen},
Proposition 1.6.25. For completeness we include the proof in the
general case. Firstly observe that $R^{\sigma}$ is a subring of
$R[x;\sigma]$, whence $R^{\sigma}[x^v]$ is the subring of
$R[x;\sigma]$ generated by $R^{\sigma}$ and $x^v$, this implies
that the elements of $R^{\sigma}[x^v]$ are polynomial in $x^v$
with coefficients in $R^{\sigma}$.

Let $p(x):=p_0+p_1x+\cdots+p_nx^n\in Z(R[x;\sigma])$, then for every $r\in R$, $rp(x)=p(x)r$, so
for every $0\leq i\leq n$, we get $rp_i=p_i\sigma^i(r)=\sigma^i(r)p_i$. If the order of $\sigma$ is
infinite, then $\sigma^i\neq i_R$ for every $i\geq 1$, hence $p_i=0$ for $i\geq 1$. Thus, in this
case $p(x)=p_0$; moreover, $p(x)$ commutes with $x$, so $p_0x=xp_0$, whence $p_0=\sigma(p_0)$,
i.e., $p_0\in R^{\sigma}$. Therefore, if $\sigma$ has infinite order, then $Z(R[x;\sigma])\subseteq
R^{\sigma}$, but clearly, $R^{\sigma}\subseteq Z(R[x;\sigma])$. Suppose that $\sigma$ has finite
order, say, $v$. If $\sigma^i\neq i_R$, i.e., if $v\nmid i$, then $p_i=0$, hence
$p(x)=p_0+p_1x^v+p_2(x^v)^2+\cdots+p_t(x^v)^t$. Since $p(x)$ commutes with $x$, then
$\sigma(p_i)=p_i$ for every $0\leq i\leq t$. This proves that $Z(R[x;\sigma])\subseteq
R^{\sigma}[x^v]$. But, $R^{\sigma}[x^v]\subseteq Z(R[x;\sigma])$ since every element of the form
$r(x^{v})^t$, with $r\in R^{\sigma}$ and $t\geq 0$, commutes with every element $s\in R$ and with
$x$.
\end{proof}


\section{Main theorem}

We start with the following easy proposition. We include the proof
for completeness.
\begin{proposition}\label{proposition3.5.1}
Let $A$ be a right Ore domain.
\begin{enumerate}
\item[\rm (i)]If $\frac{p}{q}\in Z(Q_r(A))$ then
\begin{enumerate}
\item[\rm (a)]$pq=qp$.
\item[\rm (b)]For every $s\in A-\{0\}$, $psq=qsp$.
\item[\rm (c)]$p\in Z(A)$ if and only if $q\in Z(A)$.
\end{enumerate}
\item[\rm (ii)]Let $p\in A$. Then, $\frac{p}{1}\in Z(Q_r(A))$
if and only if $p\in Z(A)$. Thus, $Z(A)\hookrightarrow Z(Q_r(A))$.
\item[\rm (iii)]If $K$ is a field and $A$ is a $K$-algebra such
that $Z(Q_r(A))=K$, then $Q(Z(A))=K=Z(Q_r(A))$.
\end{enumerate}

\end{proposition}
\begin{proof}
(i) (a) We have $\frac{p}{q}\frac{q}{1}=\frac{q}{1}\frac{p}{q}$,
so $\frac{p}{1}=\frac{qp}{q}$, whence
$\frac{p}{1}\frac{q}{1}=\frac{qp}{q}\frac{q}{1}$, i.e.,
$\frac{pq}{1}=\frac{qp}{1}$, thus $pq=qp$.

(b) For $s=0$ is clear. Let $s\in A-\{0\}$, then
$\frac{p}{q}=\frac{ps}{qs}$, so by (a), $psqs=qsps$, whence
$psq=qsp$.

(c) If $\frac{p}{q}=0$, then $\frac{p}{q}=\frac{0}{1}$ and the
claimed trivially holds. We can assume that $p\neq 0$; by (b), for
every $s\in A-\{0\}$, $psq=qsp$, hence if $p\in Z(A)$, then
$sqp=qsp$, whence $sq=qs$, i.e., $q\in Z(A)$. On the other hand,
since $\frac{q}{p}\in Z(Q_r(A))$, then if $q\in Z(A)$ we get $p\in
Z(A)$.

(ii) If $\frac{p}{1}\in Z(Q_r(A))$, then by (i), $ps=sp$ for every
$s\neq 0$, whence $p\in Z(A)$. Conversely, let $p\in Z(A)-\{0\}$
(for $p=0$, $\frac{p}{1}\in Z(Q_r(A))$), then for every
$\frac{a}{s}\in Q_r(A)$ we have
$\frac{p}{1}\frac{a}{s}=\frac{pa}{s}$ and
$\frac{a}{s}\frac{p}{1}=\frac{ac}{r}=\frac{acp}{rp}$, where
$sc=pr=rp$, with $c,r\neq 0$. From this we get
$\frac{acp}{rp}=\frac{apc}{sc}=\frac{ap}{s}=\frac{pa}{s}$, i.e.,
$\frac{p}{1}\in Z(Q_r(A))$.

(iii) From (ii), $K\subseteq Z(A)\subseteq Z(Q_r(A))=K$, so
$Z(A)=K$, and hence $Q(Z(A))=K=Z(Q_r(A))$.
\end{proof}

The next example illustrates the part (iii) of Proposition
\ref{proposition3.5.1}.

\begin{example}
We consider the quantum plane $A:=K_q[x,y]$, where $q$ is not a
root of unity. We will show that $Z(Q_r(A))=K$. Let
$\frac{p}{s}\in Z(Q_r(A))-\{0\}$, where
$p:=\sum_{i=1}^tr_ix_1^{\alpha_i}x_2^{\beta_i}$ and
$s:=\sum_{j=1}^lu_jx_1^{\theta_j}x_2^{\gamma_j}$, with $r_i,u_j\in
K-\{0\}$. From $px_1s=sx_1p$ and since $q$ is not a root of unity,
we get $\beta_i+\beta_i\theta_j=\gamma_j+\gamma_j\alpha_i$ for
every $1\leq i\leq t$ and $1\leq j\leq l$. Similarly, from
$px_2s=sx_2p$ we obtain
$\theta_j\beta_i+\theta_j=\alpha_i\gamma_j+\alpha_i$ for all
$i,j$, whence $\beta_i+\alpha_i=\gamma_j+\theta_j$, so fixing $i$
and then fixing $j$ we conclude that $p$ and $s$ are homogeneous
of the same degree (this condition is not enough since
$\frac{x_1}{x_2}\notin Z(Q_r(A))$). Now,
\begin{center}
$\frac{p}{s}\frac{x_1}{1}=\frac{x_1}{1}\frac{p}{s}$, i.e.,
$\frac{\sum_{i=1}^tr_ix_1^{\alpha_i}x_2^{\beta_i}}{\sum_{j=1}^lu_jx_1^{\theta_j-1}x_2^{\gamma_j}}=
\frac{\sum_{i=1}^tr_ix_1^{\alpha_i+1}x_2^{\beta_i}}{\sum_{j=1}^lu_jx_1^{\theta_j}x_2^{\gamma_j}}$,
\end{center}
hence there exist $c:=x_1^mp_m(x_2)+\cdots +p_0(x_2),
d:=x_1^kq_k(x_2)+\cdots +q_0(x_2)\in A-\{0\}$ such that
\begin{center}
$(\sum_{i=1}^tr_ix_1^{\alpha_i}x_2^{\beta_i})c=(\sum_{i=1}^tr_ix_1^{\alpha_i+1}x_2^{\beta_i})d$,

$(\sum_{j=1}^lu_jx_1^{\theta_j-1}x_2^{\gamma_j})c=(\sum_{j=1}^lu_jx_1^{\theta_j}x_2^{\gamma_j})d$.
\end{center}
Since $p$ and $q$ are homogeneous, we can assume
$\alpha_1>\cdots>\alpha_t$ and $\theta_1>\cdots>\theta_l$, whence
$\beta_1<\cdots<\beta_t$ and $\gamma_1<\cdots<\gamma_l$. Then,
\begin{center}
$(r_1x_1^{\alpha_1}x_2^{\beta_1})(x_1^mp_m(x_2))=r_1q^{m\beta_1}x_1^{\alpha_1+m}x_2^{\beta_1}p_m(x_2)$,

$r_1x_1^{\alpha_1+1}x_2^{\beta_1}x_1^kq_k(x_2)=r_1q^{k\beta_1}x_1^{\alpha_1+1+k}x_2^{\beta_1}q_k(x_2)$,
\end{center}
whence $\alpha_1+m=\alpha_1+1+k$, i.e., $m=k+1$. Moreover, let
$p_m$ be the leader coefficient of $p_m(x_2)$ and $q_k$ be the
leader coefficient of $q_k(x_2)$, then $q^{\beta_1}p_m=q_k$.
Similarly, we can prove that $q^{\gamma_1}p_m=q_k$, but since $q$
is not a root of unity, $\beta_1=\gamma_1$. From
$\alpha_1+\beta_1=\theta_1+\gamma_1$ we get that
$\alpha_1=\theta_1$ (considering instead the identity
$\frac{p}{s}\frac{x_2}{1}=\frac{x_2}{1}\frac{p}{s}$ we obtain the
same result). Thus, we have
\begin{center}
$\alpha_1=\theta_1$ and $\beta_1=\gamma_1$.
\end{center}
Notice that
\begin{center}
$\frac{p}{s}=r_1u_1^{-1}+\frac{p-r_1u_1^{-1}s}{s}$, with
$r_1u_1^{-1}\in K\subseteq Z(Q_r(A))$,
\end{center}
hence $\frac{p-r_1u_1^{-1}s}{s}\in Z(Q_r(A))$. But observe that
$\frac{p-r_1u_1^{-1}s}{s}=0$, contrary, we could repeat the
previous procedure and find that there exists, either $i\geq 2$
such that $\alpha_i=\theta_1=\alpha_1,\beta_i=\gamma_1=\beta_1$,
or $j\geq 2$ such that $\theta_j=\theta_1,\gamma_j=\gamma_1$, a
contradiction. Thus, $\frac{p}{s}=r_1u_1^{-1}\in K$, and hence,
$Z(Q_r(A))=K$.
\end{example}

The previous example shows that the proof of the isomorphism
$Z(Q_r(A))\cong Q(Z(A))$ by direct computation of the center of
the total ring of fractions is tedious. An alternative more
practical method is given by the following theorem.

\begin{theorem}\label{theorem6.4.2}
Let $K$ be a field and $A$ be a right Ore domain. If $A$ is a
finitely generated $K$-algebra such that ${\rm GKdim}(A)<{\rm
GKdim}(Z(A))+1$, then
\begin{center}
$Z(Q_r(A))=\{\frac{p}{q}\mid p,q\in Z(A), q\neq 0\}\cong Q(Z(A))$.
\end{center}
\end{theorem}
\begin{proof}
We divide the proof in three steps.

\textit{Step 1.} As in the proof of Theorem 4.12 in \cite{Krause},
we will show that
\begin{center}
$Q_r(A)\cong A(Z(A)_0)^{-1}=\{\frac{p}{q}\mid p\in A, q\in
Z(A)_0\}$, with $Z(A)_0:=Z(A)-\{0\}$.
\end{center}
First observe that $Z(A)_0$ is a right Ore set of $A$, so
$A(Z(A)_0)^{-1}$ exists. From the canonical injection
$Q(Z(A))\hookrightarrow A(Z(A)_0)^{-1}$, $\frac{p}{q}\mapsto
\frac{p}{q}$, we get that $A(Z(A)_0)^{-1}$ is a vector space over
$Q(Z(A))$, moreover, $A(Z(A)_0)^{-1}=AQ(Z(A))$. We will show that
the dimension of this vector space is finite. Let $V$ be a frame
that generates $A$. Since $\{V^n\}_{n\geq 0}$ is a filtration of
$A$, then $A=\bigcup_{n\geq 0}V^n$ and $AQ(Z(A))=\bigcup_{n\geq
0}V^nQ(Z(A))$. Arise two possibilities: Either there exists $n\geq
0$ such that $V^nQ(Z(A))=V^{n+1}Q(Z(A))$, or else
\begin{center}
$Q(Z(A))\subsetneq VQ(Z(A))\subsetneq V^2Q(Z(A))\subsetneq \cdots$
\end{center}
In the first case $AQ(Z(A))=V^nQ(Z(A))$ and we get the claimed. In
the second case,
\begin{center}
$\dim_{Q(Z(A))}Q(Z(A))\lneq \dim_{Q(Z(A))}VQ(Z(A))\lneq
\dim_{Q(Z(A))}V^2Q(Z(A))\lneq\cdots $
\end{center}
and we will show that this produces a contradiction. In fact, for
every $n\geq 0$,
\begin{center}
$\dim_{Q(Z(A))}V^nQ(Z(A))\geq n+1$;
\end{center}
let $u_1,\dots,u_{d(n)}$ be a $Q(Z(A))$-basis of $V^nQ(Z(A))$,
thus, $d(n)\geq n+1$; we can assume that $u_i\in V^n$ for every
$1\leq i\leq d(n)$; let $W$ be an arbitrary $K$-subspace of $Z(A)$
of finite dimension, then
\begin{center}
$(V+W)^{2n}\supseteq V^nW^n\supseteq u_1W^n\oplus \cdots \oplus
u_{d(n)}W^n$
\end{center}
(the sum is direct since the elements $u_i$ are linearly
independent over $Z(A)$); from this we get
\begin{center}
$\dim_K(V+W)^{2n}\geq d(n)\dim_k(W^n)\geq (n+1)\dim_K(W^n)$,
\end{center}
but since $V+W$ is a frame of $A$, then ${\rm GKdim}(A)\geq 1+{\rm
GKdim}(Z(A))$, false.

Now we can prove the claimed isomorphism. For this consider the
canonical injective homomorphism $g:A\to A(Z(A)_0)^{-1}$,
$a\mapsto \frac{a}{1}$. If $a\in A-\{0\}$, then $\frac{a}{1}$ is
invertible in $A(Z(A)_0)^{-1}$. In fact, the map
$h:A(Z(A)_0)^{-1}\to A(Z(A)_0)^{-1}$, $\frac{p}{q}\mapsto
\frac{a}{1}\frac{p}{q}$, is an injective $Q(Z(A))$-homomorphism
since $A$ is a domain, but as was observed above, $A(Z(A)_0)^{-1}$
is finite-dimensional over $Q(Z(A))$, therefore $h$ is surjective,
whence, there exists $\frac{p}{q}\in A(Z(A)_0)^{-1}$ such that
$\frac{a}{1}\frac{p}{q}=\frac{1}{1}$. Observe that
$\frac{p}{q}\frac{a}{1}=\frac{1}{1}$: In fact, since $p\neq 0$,
there exists $\frac{p'}{q'}$ in $A(Z(A)_0)^{-1}$ such that
$\frac{p}{1}\frac{p'}{q'}=\frac{1}{1}$, and since
$\frac{p}{q}=\frac{1}{q}\frac{p}{1}$, then
\begin{center}
$\frac{a}{1}\frac{1}{q}\frac{p}{1}\frac{p'}{q'}=\frac{1}{1}\frac{p'}{q'}$,
i.e., $\frac{a}{1}\frac{1}{q}=\frac{p'}{q'}$, so
$\frac{a}{1}\frac{1}{q}\frac{q}{1}=\frac{p'}{q'}\frac{q}{1}$,
whence $\frac{a}{1}=\frac{p'}{q'}\frac{q}{1}$, so
$\frac{p}{q}\frac{a}{1}=\frac{p}{q}\frac{p'}{q'}\frac{q}{1}=\frac{1}{q}\frac{p}{1}\frac{p'}{q'}\frac{q}{1}=\frac{1}{q}\frac{q}{1}=\frac{1}{1}$.
\end{center}
In order to conclude the proof of the isomorphism
$A(Z(A)_0)^{-1}\cong Q_r(A)$, observe that any element
$\frac{p}{q}\in A(Z(A)_0)^{-1}$ can be written as
$\frac{p}{q}=g(p)g(q)^{-1}$.

\textit{Step 2.} Let $C:=\{\frac{p}{q}\mid p,q\in Z(A), q\neq
0\}$. If $\frac{p}{q}\in Z(Q_r(A))$, then by the first step we can
assume that $q\in Z(A)_0$, and from the part (i)-(c) of
Proposition \ref{proposition3.5.1}, we get $p\in Z(A)$. Therefore,
$Z(Q_r(A))\subseteq C$. Conversely, let $\frac{p}{q}\in C$, then
$p,q\in Z(A)$, with $q\neq 0$, whence, by the part (ii) of
Proposition \ref{proposition3.5.1}, $\frac{p}{1},\frac{q}{1}\in
Z(Q_r(A))$, so $\frac{1}{q}\in Z(Q_r(A))$, and hence,
$\frac{p}{q}=\frac{p}{1}\frac{1}{q}\in Z(Q_r(A))$. Thus,
$C\subseteq Z(Q_r(A))$.

\textit{Step 3}. According to the part (ii) of Proposition
\ref{proposition3.5.1}, we have the canonical injective
homomorphism $Z(A)\xrightarrow{\iota} Z(Q_r(A))$, $p\mapsto
\frac{p}{1}$, that sends invertible elements of $Z(A)$ in
invertible elements of $Z(Q_r(A))$, moreover, by the step 2, every
element $\frac{p}{q}\in Z(Q_r(A))$ can be written
$\frac{p}{q}=\iota(p)\iota(q)^{-1}$. This proves the isomorphism
$Z(Q_r(A))\cong Q(Z(A))$.
\end{proof}

\section{Examples}

Next we present many examples of $K$-algebras that satisfy the
hypotheses of Theorem \ref{theorem6.4.2}. Most of them within the
skew $PBW$ extensions.

\begin{example}\label{example6.4.3}
(i) Any domain $A$ such that $\dim_K A<\infty$. For example, the
real algebra $\mathbb{H}$ of quaternions since
$\dim_{\mathbb{R}}(\mathbb{H})=4$.

(ii) Any right Ore domain $A$ finitely generated as $Z(A)$-module.
\end{example}

\begin{example}
Applying Propositions \ref{Orespbw} and \ref{lezamareyes5.1}, we
will check next that the following skew $PBW$ extensions are
$K$-algebras that satisfy the hypotheses of Theorem
\ref{theorem6.4.2}. The precise definition of any of these
algebras can be found in \cite{Lezama-reyes}.

(i) Consider a skew polynomial ring $A:=R[x;\sigma]$, with $R$ a
commutative domain $R$ that is a $K$-algebra generated by a
subspace $V$ of finite dimension such that $\sigma(V)\subseteq V$,
$\sigma$ is $K$-linear of finite order $m$, $R^\sigma=K$ and ${\rm
GKdim}(R)=0$. Then, ${\rm GKdim}(A)=1$, and from Proposition
\ref{proposition1.2.8}, $Z(A)=K[x^m]$, and hence, ${\rm
GKdim}(Z(A))=1$. Thus, $Z(Q_r(A))\cong Q(Z(A))=K(x^m)$. A
particular case of this general example is
$A:=\mathbb{C}[x;\sigma]$ as $\mathbb{R}$-algebra, with
$\sigma(z):=\overline{z}$, $z\in \mathbb{C}$ (here $\mathbb{C}$
and $\mathbb{R}$ are the fields of complex and real numbers,
respectively). In this case the order of $\sigma$ is two and ${\rm
GKdim}(\mathbb{C})=0$ since $\dim_\mathbb{R}(\mathbb{C})=2$.

(ii) Let $char(K)=p>0$ and $A:=S_h:=K[t][x_h;\sigma_h]$ be the
algebra of shift operators. Then, ${\rm GKdim}(A)=2$. Moreover,
for every $k\geq 0$, $\sigma_h^k(t)=t-kh$, then $\sigma_h^p(t)=t$,
i.e., the order of $\sigma_h$ is $p$, therefore,
$Z(A)=K[t]^{\sigma_h}[x_h^p]$. Since $K[t^p]\subseteq
K[t]^{\sigma_h}\subseteq K[t]$ and $K[t]$ is finitely generated
over $K[t^p]$, then ${\rm GKdim}(K[t^p])={\rm GKdim}(K[t])=1$,
whence ${\rm GKdim}(K[t]^{\sigma_h})=1$. Therefore, ${\rm
GKdim}(Z(A))=2$. Thus, $Z(Q_r(A))\cong
Q(Z(A))=Q(K[t]^{\sigma_h})(x_h^p)$.

(iii) Let $char(K)=p>0$ and
$A:=K[t][x;\frac{d}{dt}][x_h;\sigma_h]$ be the algebra of shift
differential operators. Then, ${\rm GKdim}(A)=3$ and can be proved
that $Z(A)=K[x^p,x_h^p, t^{p^2}-t^p]$. Since ${\rm
GKdim}(Z(A))=3$, then $Z(Q_r(A))\cong Q(Z(A))=K(x^p,x_h^p,
t^{p^2}-t^p)$.

(iv) Let $char(K)=p>0$ and $A:=A_n(K)$ be the Weyl algebra. Since
${\rm GKdim}(A)=2n$ and
$Z(A)=K[t_1^p,\dots,t_n^p,x_1^p,\dots,x_n^p]$ (see
\cite{Levandovskyy}, ejemplo 1.3.), then ${\rm GKdim}(Z(A))=2n$.
Therefore, $Z(Q_r(A))\cong Q(Z(A))=
K(t_1^p,\dots,t_n^p,x_1^p,\dots,x_n^p)$.

(v) Let $char(K)=p>0$ and $A:=\mathcal{J}:=K\{x,y\}/\langle
yx-xy-x^2\rangle$ be the Jordan algebra. Since ${\rm
GKdim}(\mathcal{J})=2$ and $Z(A)=K[x^p,y^p]$ (see Theorem 2.2 in
\cite{Shirikov}), then ${\rm GKdim}(Z(A))=2$, whence
$Z(Q_r(A))\cong Q(Z(A))=K(x^p,y^p)$.

(vi) Consider the quantum plane $A:=K_q[x,y]$, with $q\neq 1$ a
root of unity of degree $m\geq 2$. Then ${\rm GKdim}(A)=2$ and
$Z(A)=K[x^m,y^m]$ (see \cite{Shirikov}). Therefore,
$Z(Q_r(A))\cong Q(Z(A))=K(x^m,y^m)$.

(vii) The previous example can be extended to the quantum
polynomials $A:=K_q[x_1,\dots,x_n]$, where $n\geq 2$ and $q\in
K-\{0,1\}$, defined by
\begin{center}
$x_jx_i=qx_ix_j$, with $1\leq i<j\leq n$.
\end{center}
If $q$ is a root of unity of degree $m\geq 2$, then can be proved
that if $n$ even, then $Z(A)=K[x_1^m,\dots,x_n^m]$. Therefore,
${\rm GKdim}(Z(A))=n={\rm GKdim}(A)$ and hence
\begin{center}
$Z(Q_r(A))\cong Q(Z(A))=K(x_1^m,\dots,x_n^m)$.
\end{center}

(viii) Let $A_q$ be the quantum Weyl algebra generated by $x,y$
with rule of multiplication $yx=qxy+a$, where $q,a\in K-\{0\}$. If
$q$ is a primitive root of unity of degree $m\geq 2$, then
$Z(A_q)=K[x^m,y^m]$ (see \cite{Zhangetal3}). Since ${\rm
GKdim}(A_q)=2$, then $Z(Q_r(A_q))\cong Q(Z(A))=K(x^m,y^m)$.

(ix) In \cite{LezamaHelbert2} has been computed the center of the
following algebras. In every example we assume that the parameters
$q$'s are root of unity of degree $l\geq 2$, or $l_i\geq 2$,
appropriately:
\begin{enumerate}
\item[\rm (a)]Algebra of $q$-differential operators, then $Z(A)=K[x^{l},y^{l}]
$ and ${\rm GKdim}(A)=2$, so $Z(Q_r(A))\cong
Q(Z(A))=K(x^{l},y^{l})$.
\item[\rm (b)]Additive analogue of the Weyl algebra, $ Z(A)=K[x_{1}^{l_{1}},\dots,x_{n}^{l_{n}},y_{1}^{l_{1}},\dots,y_{n}^{l_{n}}]
$ and ${\rm GKdim}(A)=2n$, so $Z(Q_r(A))\cong
Q(Z(A))=K(x_{1}^{l_{1}},\dots,x_{n}^{l_{n}},y_{1}^{l_{1}},\dots,y_{n}^{l_{n}})$.
\item[\rm (c)]Algebra of linear partial $q$-dilation operators, in this case we have
${\rm GKdim}(A)=2n$ and
$Z(A)=K[t_{1}^{l},\dots,t_{n}^{l},H_{1}^{l},\dots,H_{n}^{l}] $.
Therefore, $Z(Q_r(A))\cong
Q(Z(A))=K(t_{1}^{l},\dots,t_{n}^{l},H_{1}^{l},\dots,H_{n}^{l})$.

\item[\rm (d)]Algebra of linear partial $q$-differential operators,
in this case we have ${\rm GKdim}(A)=2n$ and
$Z(A)=K[t_{1}^{l},\dots,t_{n}^{l},D_{1}^{l},\dots,D_{n}^{l}] $.
Hence, $Z(Q_r(A))\cong
Q(Z(A))=K(t_{1}^{l},\dots,t_{n}^{l},D_{1}^{l},\dots,D_{n}^{l})$.
\end{enumerate}

(x) Let $\mathfrak{sl}(n,K)$ be the Lie algebra of $2\times 2$
matrices with null trace with $K$-basis $e,f,h$. If $char(K)=2$,
then $Z(\mathcal{U}(\mathfrak{sl}(2,K)))=K[e^2,f^2,h]$ (see
\cite{Levandovskyy}, p. 147). Moreover, ${\rm
GKdim}(\mathcal{U}(\mathfrak{sl}(2,K)))=3$. Thus, $Z(Q_r(A))\cong
Q(Z(A))=K(e^2,f^2,h)$.
\end{example}

\begin{remark}
As occurs for the Gelfand-Kirillov conjecture (see \cite{GK}), if
the hypotheses of Theorem \ref{theorem6.4.2} fail, then the
isomorphism $Z(Q_r(A))\cong Q(Z(A))$ could hold or fail. Thus, the
hypotheses are not necessary conditions. For example,

(a) $\mathbb{H}$ is not finitely generated as
$\mathbb{Q}$-algebra, however $Q_r(\mathbb{H})=\mathbb{H}$ and
$Z(Q_r(\mathbb{H}))\cong \mathbb{R}\cong Q(Z(\mathbb{H}))$.

(b) Let $K$ be a field with $char(K)=0$, and let $\mathcal{G}$ be
a three-dimensional completely solvable Lie algebra with basis
$x,y,z$ such that $[y,x]=y$, $[z,x]=\lambda z$ and $[y,z]=0$,
$\lambda\in K-\{0\}$ (see Example 14.4.2 in \cite{McConnell}). If
$\lambda\in K-\mathbb{Q}$, then $Z(U(\mathcal{G}))=K$ and ${\rm
GKdim}(\mathcal{U}(\mathcal{G}))=3$, thus, in this case ${\rm
GKdim}(\mathcal{U}(\mathcal{G}))\nless{\rm
GKdim}(Z(\mathcal{U}(\mathcal{G})))+1$, and 14.4.7 in
\cite{McConnell} shows that
$Z(Q_r(\mathcal{U}(\mathcal{G})))\ncong
Q(Z(\mathcal{U}(\mathcal{G})))$.

(c) Let $A:=U_q^{+}(sl_m)$ be the quantum enveloping algebra of
the Lie algebra of strictly superior triangular matrices of size
$m\times m$ over a field $K$, where $q\in K-\{0\}$ is not a root
of unity. In \cite{Alev2} was proved that $Z(A)$ is the classical
commutative polynomial algebra over $K$ in $n$ variables, with
$m=2n$ or $m=2n+1$, whence, ${\rm GKdim}(Z(A))=n$. On the other
hand, according to \cite{Alev2}, p.\,236, $A$ is an iterated skew
polynomial ring of $K$ of $\frac{m(m-1)}{2}$ variables, hence
${\rm GKdim}(A)=\frac{m(m-1)}{2}$. Thus, ${\rm GKdim}(A)\nless
{\rm GKdim}(Z(A))+1$, however, $Z(Q_r(A))\cong Q(Z(A))$.
\end{remark}


\end{document}